\documentclass[letterpaper,10pt,reqno,onefignum,onetabnum]{amsart}
\usepackage[english]{babel}
\usepackage{amsmath}
\usepackage{amsthm}
\usepackage{verbatim}
\usepackage{mathrsfs}
\usepackage{bm} 
\usepackage[dvipsnames]{xcolor}
\usepackage{hyperref}
\hypersetup{
     colorlinks = true,
     linkcolor = OliveGreen,
     anchorcolor = OliveGreen,
     citecolor = OliveGreen,
     filecolor = OliveGreen,
     urlcolor = OliveGreen
     }
\usepackage{algorithm}
\usepackage{algorithmic}
\usepackage{float}
\usepackage{lipsum}
\usepackage{amsfonts}
\usepackage{amssymb}
\usepackage{graphicx}
\usepackage{epstopdf}
\usepackage{cases}
\usepackage{multirow}
\ifpdf
  \DeclareGraphicsExtensions{.eps,.pdf,.png,.jpg}
\else
  \DeclareGraphicsExtensions{.eps}
\fi
\usepackage{verbatim}
\usepackage{mathrsfs}
\usepackage{bm}
\usepackage{enumitem}
\newcommand{\conv}{\ensuremath{\text{\rm conv}\,}}
\newtheorem{teo}{Theorem}[section]
\newtheorem{prop}[teo]{Proposition}
\newtheorem{lem}[teo]{Lemma}

\newtheorem{pro}[teo]{Problem}

\newtheorem{remark}[teo]{Remark}

\usepackage{lipsum}
\usepackage{amsfonts}
\usepackage{amssymb}
\usepackage{graphicx}
\usepackage{epstopdf}
            
\usepackage{cases}
\usepackage{multirow}
\usepackage{algorithmic}
\usepackage{tikz}
\usetikzlibrary{matrix}
\usetikzlibrary{arrows}
\ifpdf
  \DeclareGraphicsExtensions{.eps,.pdf,.png,.jpg}
\else
  \DeclareGraphicsExtensions{.eps}
\fi
\usepackage{verbatim}
\usepackage{mathrsfs}
\usepackage{bm} 
\usepackage{color}
\usepackage{caption}
\usepackage{subcaption}

\newcommand{\N}{\mathbb N}

\newcommand{\R}{\mathbb R}

\renewcommand{\H}{\mathcal{H}}
\newcommand{\G}{\mathcal G}

\newcommand{\HH}{{\bm{\mathcal{H}}}}

\newcommand{\id}{\textnormal{Id}}
\newcommand{\lev}{\textnormal{lev}}
\newcommand{\inte}{\ensuremath{\operatorname{int}}}
\newcommand{\interior}{\textnormal{int}}

\newcommand{\s}{\sigma}

\newcommand{\weak}{\rightharpoonup}
\newcommand{\ran}{\textnormal{ran}\,}
\newcommand{\dom}{\textnormal{dom}\,}

\newcommand{\bdry}{\ensuremath{\text{\rm bdry}\,}}
\newcommand{\zer}{\textnormal{zer}}

\newcommand{\gra}{\textnormal{gra}\,}

\newcommand{\argm}[1]{\underset{#1}{\argmin\, }}

\newcommand{\Scal}[2]{{\bigg\langle{{#1}\:\bigg |~{#2}}\bigg\rangle}}
\newcommand{\scal}[2]{{\left\langle{{#1}\mid{#2}}\right\rangle}}
 
\newcommand{\pnorm}[1]{|\hspace{-.3mm}|\hspace{-.3mm}|{#1}|\hspace{-.3mm}|\hspace{-.3mm}|}
 
\newcommand{\menge}[2]{\big\{{#1}~\big |~{#2}\big\}} 
\newcommand{\pinf}{\ensuremath{{+\infty}}}

\newcommand{\RP}{\ensuremath{\left[0,+\infty\right[}}
\newcommand{\RM}{\ensuremath{\left]-\infty,0\right]}}

\newcommand{\RPP}{\ensuremath{\left]0,+\infty\right[}}

\newcommand{\RX}{\ensuremath{\left]-\infty,+\infty\right]}}
\newcommand{\RXX}{\ensuremath{\left[-\infty,+\infty\right]}}
\newcommand{\sri}{\ensuremath{\text{\rm sri}\,}}
\newcommand{\prox}{\ensuremath{\text{\rm prox}}}
\newcommand{\weakly}{\ensuremath{\:\rightharpoonup\:}}
\usepackage{geometry}
\numberwithin{equation}{section}

\DeclareFontEncoding{FMS}{}{}
\DeclareFontSubstitution{FMS}{futm}{m}{n}
\DeclareFontEncoding{FMX}{}{}
\DeclareFontSubstitution{FMX}{futm}{m}{n}
\DeclareSymbolFont{fouriersymbols}{FMS}{futm}{m}{n}
\DeclareSymbolFont{fourierlargesymbols}{FMX}{futm}{m}{n}
\DeclareMathDelimiter{\nr}{\mathord}{fouriersymbols}{152}{fourierlargesymbols}{147}

\DeclareMathOperator*{\argmin}{arg\,min}
\DeclareMathDelimiter{\nr}{\mathord}{fouriersymbols}{152}{fourierlargesymbols}{147}

\title[{Four Operator Splitting}]{Four Operator Splitting via a 
Forward-Backward-Half-Forward Algorithm with Line Search}

\author{Luis M. Brice\~no-Arias \& Fernando Rold\'an}
\address{Departamento de Matem\'{a}tica, Universidad T\'{e}cnica Federico Santa Mar\'{i}a, Avenida Espa\~{n}a 1680, Valpara\'{i}so, Chile}
\email{luis.briceno@usm.cl, fernando.roldan@usm.cl}

\subjclass[2010]{47H05, 47H10, 65K05, 65K15, 90C25, 49M29.}

\begin{document}

\begin{abstract}In this article we provide a splitting method for 
	solving monotone inclusions in a real Hilbert space involving 
	four operators: a maximally monotone, a monotone-Lipschitzian, a
	cocoercive, and a monotone-continuous operator.  The proposed 
	method takes 
	advantage of the intrinsic properties of each operator, 
	generalizing 
	the 
	forward-back-half forward splitting and the Tseng's algorithm 
	with line-search. At each iteration, our 
	algorithm defines the step-size by using a line search in which 
	the 
	monotone-Lipschitzian and the cocoercive operators need only one 
	activation. We also derive 
	a method for solving non-linearly  constrained composite convex 
	optimization problems in real Hilbert spaces. Finally, we 
	implement 
	our algorithm in a non-linearly constrained least-square 
	problem 
	and 
	we compare its performance with available methods in the 
	literature.
\par
\bigskip

\noindent \textbf{Keywords.} {\it Convex Optimization, Cocoercive 
Operator, Lipchitizian Operator, Monotone Operator Theory, Splitting 
Algorithms.}
\end{abstract}

\maketitle

\section{Introduction}\label{Intro}

In this paper, we aim at finding a zero of the sum of a maximally 
monotone, a cocoercive, a monotone-Lipschitzian, and a 
monotone-continuous operators, which is also in a convex closed 
subset $X$ of a Hilbert space.
This inclusion encompasses several problems in partial 
differential equations coming from mechanical models 
\cite{Gabay1983,Glowinski1975,Goldstein1964}, differential 
inclusions \cite{AubinHelene2009,Showalter1997}, game theory 
\cite{Nash13}, 
among other disciplines. When $X$ is the whole Hilbert space,
the algorithms proposed in
\cite{Bot-2019-ACM,Bot2016NA,Bot2013,Bot2015AMC,BotHendrich2013,briceno2011SIAM,CombettesMinh2022,Cevher2020SVVA,Combettes2018MP,CombPes12,Combettes2014Optimization,Csetnek2019AMO,davis2015,Dong2021,DungVu15,Eckstein2017,Eckstein2020,EcksteinJohnstone,Malitsky2020SIAMJO,Raguet-SIAM-2013,Rieger2020AMC,RyuVu2020,Vu13}
can solve this kind of problems  under additional assumptions 
or without exploiting the intrinsic properties of the involved 
operators.
Indeed, the algorithms in 
\cite{Bot2013,Bot2015AMC,BotHendrich2013,briceno2011SIAM,Combettes2018MP,Eckstein2017}
need to compute the resolvents of all the monotone operators, which 
are not explicit in general or they can be 
numerically expensive. The schemes proposed in  
\cite{Bot2016NA,CombPes12,DungVu15,Eckstein2020} take 
advantage of the 
properties of the monotone-Lipschitzian operator, but the 
cocoercivity and the continuity of the others are not leveraged. In 
fact, 
the algorithms in 
\cite{Bot2016NA,CombPes12,DungVu15,Eckstein2020} need to 
activate the continuous operator via its resolvent and to explicitly 
implement the cocoercive and the monotone-Lipschitzian operators 
twice by iteration. In contrast, the algorithms in 
\cite{Cevher2020SVVA,Csetnek2019AMO,Malitsky2020SIAMJO,Rieger2020AMC,RyuVu2020}
activate the cocoercive and the monotone-Lipschitzian 
operators only once by iteration, but they need to store in 
the 
memory the two past iterations and the step-size is reduced 
significantly. Furthermore, the methods in 
\cite{Cevher2020SVVA,Csetnek2019AMO,Malitsky2020SIAMJO} consider only 
one maximally monotone operator and, hence, it needs to compute 
the resolvent of the sum of the maximally monotone and the 
monotone-continuous operator. On the other hand, methods in 
\cite{Rieger2020AMC,RyuVu2020} need to calculate the resolvent of 
the monotone-continuous operator, which is not simple in general.
In addition, methods proposed in 
\cite{Bot-2019-ACM,Combettes2014Optimization,davis2015,Dong2021,EcksteinJohnstone,Raguet-SIAM-2013,Vu13}
take advantage of the cocoercive operator by activating it once by 
iteration, but they do not exploit continuity nor the 
monotone-Lipschitzian
property of the operators and they need to compute their resolvents. 
The method in \cite{CombettesMinh2022} exploits the properties of 
the cocoercive and the monotone-Lipschitzian operators, but it does 
not take into advantage the monotonicity and continuity of one of the 
operators and need to compute its resolvent. 
Other methods solving this kind of problems including a normal cone 
to a closed vector subspace, 
when the continuous operator is zero and either the cocoercive or the 
monotone-Lipschitzian operator is zero, are discussed in 
\cite{briceno2015Optim,Briceno2015JOTA,Spingran1983AMO}. 

In the case when $X$ is a subset of the domain of the maximally 
monotone operator, 
methods exploiting the monotonicity and continuity property of one of 
the operators involved in our problem are proposed in 
\cite{BricenoDavis2018,Tseng2000SIAM}. In particular, the algorithm 
in  \cite{Tseng2000SIAM} solves our problem by using a line search
procedure involving the sum of the cocoercive, the 
monotone-Lipschitzian, and the 
monotone-continuous operators. On the other hand, the 
\textit{forward-backward-half forward splitting} (FBHF) 
method, proposed in \cite{BricenoDavis2018}, solves
our problem via a line search which activates the sum of the 
the monotone-Lipschitzian and the monotone-continuous 
operators.
As perceived in \cite{BricenoDavis2018}, to reduce the activation of 
monotone operators in the line search procedure can reduce the 
number of iterations significantly (see, e.g., 
\cite[Table~3]{BricenoDavis2018} in which this reduction is around a 
20\%). Moreover, when the explicit implementation of these operators 
are expensive (e.g., high dimensional problems), the computational 
time can be much larger because line search procedures need to 
activate those operators several times by iteration. 

In this paper, we propose a fully split method for solving our 
monotone 
inclusion, which take advantage of the intrinsic properties of each 
operator involved in the inclusion. More precisely, the proposed 
method activates the cocoercive operator once, the 
monotone-Lipschitzian operator twice by iteration, and uses a line 
search only implementing the monotone-continuous operator. We 
also ensure the weak convergence of our algorithm under 
hypotheses on the set $X$ which are
independent of the domain of the maximally monotone operator,
generalizing some results in
\cite{BricenoDavis2018,Tseng2000SIAM}.
We explore an interesting example in optimization, involving 
non-linear inequality constraints governed by convex G\^ateaux 
differentiable functions. We provide conditions on this function in 
order 
to guarantee that the saddle operator obtained from the Lagrangian of 
the problem is monotone and continuous and satisfies the hypotheses 
of our main convergence theorem. Finally, we 
provide a numerical experiment which illustrate the efficiency of our 
algorithm.

The paper is organized as follows. In Section~\ref{sec:2} we set our 
notation. In Section~\ref{sec:3} we provide some technical lemmas, 
our splitting  method for solving Problem~\ref{prob:problem1}, and our
convergence result. In Section~\ref{sec:4} we derive an algorithm for 
solving a non-linearly constrained composite convex
optimization problem. Finally, in Section~\ref{sec:5} we provide a 
numerical experiment illustrating the efficiency of the 
method proposed 
in Section~\ref{sec:4}.

\section{Preliminaries}\label{sec:2}
Throughout this paper $\H$ and $\G$ are real Hilbert spaces. We 
denote the scalar 
product by $\scal{\cdot}{\cdot}$ and the associated norm by 
$\|\cdot \|.$ 
The symbols $\weakly$ and $\to$ denote the weak and strong 
convergence, respectively.
Given a linear bounded 
operator $M:\H \to \G$, we denote its adjoint by $M^*\colon\G\to\H$. 
$\id$ denotes the identity operator on $\H$. 
Let $D\subset \H$ be non-empty, let $T: D \rightarrow \H$, and let 
$\beta 
\in \left]0,+\infty\right[$. 
The operator $T$ is 
$\beta$-cocoercive if 
\begin{equation} \label{def:coco}
	(\forall x \in D) (\forall y \in D)\quad \langle x-y \mid Tx-Ty 
	\rangle \geq 
	\beta \|Tx - Ty 
	\|^2
\end{equation}
and it is $\beta$-Lipschitzian if 
\begin{equation} \label{def:Lips}
	(\forall x \in D) (\forall y \in D)\quad \|Tx-Ty\| \leq  \beta\|x 
	- y 
	\|.
\end{equation}

Let $A:\H \rightarrow 2^{\H}$ be a set-valued operator. The domain, 
range, graph, and the zeros of $A$ are, respectively, $\dom\, A = 
\menge{x \in \H}{Ax \neq  \varnothing}$, $\ran\, A = \menge{u \in 
	\H}{(\exists x \in \H)\, u \in Ax}$, $\gra A = \menge{(x,u) \in 
	\H 
	\times \H}{u \in Ax}$, and $\zer A = \menge{x \in \H}{0 \in Ax}$.
The inverse of $A$ is $A^{-1}\colon \H  \to 
2^\H \colon u 
\mapsto 
\menge{x \in \H}{u \in Ax}$ and the resolvent of $A$ is 
$J_A=(\id+A)^{-1}$. 
The operator $A$  is monotone if 
\begin{equation}\label{def:monotone}
	(\forall (x,u) \in \gra A) (\forall (y,v) \in \gra A)\quad 
	\scal{x-y}{u-v} \geq 
	0.
\end{equation}
Moreover, $A$ is maximally monotone if it is monotone and there 
exists no 
monotone operator $B :\H\to  2^{\H}$ such that $\gra B$ properly 
contains $\gra A$, i.e., for every $(x,u) \in \H \times \H$,
\begin{equation} \label{def:maxmonotone}
	(x,u) \in \gra A \quad \Leftrightarrow\quad  (\forall (y,v) \in 
	\gra A)\ \ 
	\langle x-y \mid u-v \rangle \geq 0.
\end{equation}
$A$ is locally bounded at $x \in\H$, if there exists $\delta \in 
\RPP$ such that $A(\mathcal{B}(x;\delta))$ is bounded, where 
$\mathcal{B}(x;\delta)$ stands for the ball centered at $x$ with 
radius $\delta$. Moreover, $A$ is 
locally bounded in $\varnothing \neq D \subset \H$ if, for every $x 
\in D$, $A$ is locally bounded at $x$.

We denote by $\Gamma_0(\H)$ the class of proper lower 
semicontinuous convex functions $f\colon\H\to\RX$. Let 
$f\in\Gamma_0(\H)$.
The Fenchel conjugate of $f$ is 
defined by $f^*\colon u\mapsto \sup_{x\in\H}(\scal{x}{u}-f(x))$, 
which 
is a function in $\Gamma_0(\H)$. The subdifferential of $f$ is the 
maximally monotone operator
$$\partial f\colon \H \to 2^\H \colon x\mapsto 
\menge{u\in\H}{(\forall y\in\H)\:\: 
	f(x)+\scal{y-x}{u}\le f(y)}.$$
It turns out that $(\partial f)^{-1}=\partial f^*$
and that $\zer\partial f$ is the set of 
minimizers of $f$, which is denoted by $\arg\min_{x\in \H}f(x)$. 
We denote the proximity operator of $f$ by
\begin{equation}
	\label{e:prox}
	\prox_{f}\colon 
	x\mapsto\argm{y\in\H}\left(f(y)+\frac{1}{2}\|x-y\|^2\right).
\end{equation}
We have 
$\prox_f=J_{\partial f}$. Moreover, it follows from \cite[Theorem 
14.3]{bauschkebook2017} that
\begin{equation}
	\big(\forall \gamma >0) \quad	
	\label{e:Moreau_nonsme}
	\prox_{\gamma f}+\gamma \prox_{f^*/\gamma} \circ 
	\id/\gamma=\id.
\end{equation}
We denote by 
$\lev_{<0}f=\menge{x\in\H}{f(x)<0}$ and by
$\lev_{\le0}f=\menge{x\in\H}{f(x)\le0}$.
Given a non-empty closed convex set $C\subset\H$, we denote by 
$P_C$ the projection onto $C$ and by
$\iota_C\in\Gamma_0(\H)$ the indicator function of $C$, which 
takes the value $0$ in $C$ and $\pinf$ otherwise. 
For further properties of monotone operators,
nonexpansive mappings, and convex analysis, the 
reader is referred to \cite{bauschkebook2017}.

\section{Main Results}\label{sec:3}
We aim at solving the following monotone inclusion 
problem.
\begin{pro}\label{prob:problem1}
	Let $X$ be a nonempty closed convex subset of a real Hilbert 
	space 
	$\H$, let $A: \H \to 2^\H$ be a  maximally monotone operator, let 
	$B_1 : \H \to \H$ be a $\beta$-cocoercive operator for some 
	$\beta 
	>0$, let $B_2 :\H \to \H$ be a monotone and $L$-Lipschitzian 
	operator for some $L>0$, and let $B_3 : \H \to 2^\H$ be a 
	maximally 
	monotone operator such that $B_3$ is single valued and 
	continuous 
	in $\dom A \cup X \subset \dom B_3$. Moreover assume that 
	$A+B_3$ 
	is 
	maximally 
	monotone. The 
	problem is to 
	\begin{equation}\label{eq:problem1}
		\text{find} \quad x \in X \quad \text{such that} \quad 0 \in 
		Ax+B_1x+B_2x+B_3x,
	\end{equation}
	under the assumption that the set of solutions to 
	\eqref{prob:problem1} is nonempty.
\end{pro}

We first study some properties of the monotone 
operators involved in Problem~\ref{prob:problem1}, which ensure the 
finite termination of the backtracking procedure in our method. 

\begin{lem}\label{prop:varphiz}
	In the context of Problem~\ref{prob:problem1}, let $z$ and $y$ in 
	$\H$, and define 
	\begin{equation}\label{eq:varphiz}
		(\forall \gamma > 0) \quad \quad	x_{z,y}(\gamma) = 
		J_{\gamma A} 
		( z - 
		\gamma y) \quad \text{and} 
		\quad \varphi_{z,y}( \gamma) = \frac{\| z- 
			x_{z,y}(\gamma)\|}{\gamma}.
	\end{equation}
	Then, the following statements holds:
	\begin{enumerate}
		\item\label{prop:varphiz1}
		$\varphi_{z,y}$ is nonincreasing.
		\item\label{prop:varphiz2}
		$(\forall z \in \dom A) \displaystyle  \quad \lim_{ \gamma 
			\downarrow 0} 
		\varphi_{z,y}(\gamma) = \| (A+y)^0 z\| = \min_{w \in 
			Az+y}\|w\|
		$.
		\item \label{prop:varphiz23} Set
		\begin{equation}\label{eq:defsetC}
			\mathcal{C}=\menge{z \in \H}{ \lim_{\gamma \downarrow 0} 
				\varphi_{z,y}(\gamma)  < +\infty}.
		\end{equation} 
		Then, $\dom A \subset \mathcal{C} 
		\subset \overline{\dom} A$.
		\item\label{prop:varphiz3} Suppose that one of the following 
		holds:
		\begin{enumerate}
			\item\label{prop:varphiz31} $z \in \mathcal{C}$.
			
			\item\label{prop:varphiz32} $z \in \dom B_3 
			\setminus\mathcal{C}$,  
			$y=(B_1+B_2+B_3)z$, and $B_3$ is locally bounded at 
			$P_{\overline {\dom} A} z$. 		
			\item\label{prop:varphiz33}  $z \in \dom B_3 
			\setminus\mathcal{C}$, 
			$y=(B_1+B_2+B_3)z$, and $\overline{\dom} A \subset \inte 
			\dom 
			B_3$.
			
		\end{enumerate}
		Then, for every 
		$\theta \in \left]0,1\right[$, there exists $\gamma(z) >0$ 
		such 
		that, for 
		every  $\gamma \in \left]0, \gamma(z)\right]$,
		\begin{equation}\label{eq:destheta}
			\gamma \| B_3z - B_3 x_{z,y} (\gamma) \| \leq \theta \|z 
			- x_{z,y} 
			(\gamma)\|.
		\end{equation}

	\end{enumerate}
\end{lem}

\begin{proof}
	
	Let $z \in \H$. Note that, if $z \in \zer ( A + y )$,
	
	\begin{align}
		(\forall \gamma > 0 )\quad 0 \in Az+y \quad &\Leftrightarrow  
		\quad  z-\gamma y \in \gamma A z + z 
		\nonumber \\ 
		&\Leftrightarrow  \quad  z = x_{z,y} (\gamma)\nonumber \\ 
		&\Leftrightarrow  \quad  \varphi_{z,y}(\gamma) = 0 
		\label{varphiz=0}.   
	\end{align} 
	In this case, \ref{prop:varphiz1}, \ref{prop:varphiz2}, and 
	\ref{prop:varphiz3} 
	are clear. Henceforth, assume $z \in \H \setminus \zer ( A + y 
	)$. 
	It follows from 
	\eqref{eq:varphiz} 
	that
	\begin{equation}\label{eq:contencionxzy}
		\frac{z-x_{z,y}(\gamma)}{\gamma}  - y \in A 
		{x_{z,y}(\gamma)}. 
	\end{equation}
	\ref{prop:varphiz1}: For every $\gamma_1$ 
	and $\gamma_2$ in $\RPP$, \eqref{eq:contencionxzy}  
	and the 
	monotonicity of $A$ yield
	\begin{align*}
		0 &\leq \Scal{\frac{z-x_{z,y}(\gamma_1)}{\gamma_1}  
			-\frac{z-x_{z,y}(\gamma_2)}{\gamma_2} 
		}{x_{z,y}(\gamma_1)-x_{z,y}(\gamma_2)}\\
		& = \Scal{\frac{z-x_{z,y}(\gamma_1)}{\gamma_1}  
			-\frac{z-x_{z,y}(\gamma_2)}{\gamma_2} 
		}{x_{z,y}(\gamma_1)-z-(x_{z,y}(\gamma_2)-z)}\\
		& = - \frac{1}{\gamma_1} \|z-x_{z,y}(\gamma_1)\|^2 + \left( 
		\frac{1}{\gamma_1}+\frac{1}{\gamma_2} \right) 
		\scal{z-x_{z,y}(\gamma_1)}{z-x_{z,y}(\gamma_2)}- 
		\frac{1}{\gamma_2} \|z-x_{z,y}(\gamma_2)\|^2.
	\end{align*} 
	Hence, we obtain
	
	\begin{align*}
		\gamma_1 \varphi_{z,y}(\gamma_1)^2+\gamma_2 
		\varphi_{z,y}(\gamma_2)^2 &\leq ( \gamma_1+\gamma_2) \Scal{ 
			\frac{z-x_{z,y}(\gamma_1)}{\gamma_1}}{\frac{z-x_{z,y}(\gamma_2)}{\gamma_2}}\\
		&\leq \frac{( \gamma_1+\gamma_2)}{2} 
		(\varphi_{z,y}(\gamma_1)^2+\varphi_{z,y}(\gamma_2)^2),
	\end{align*}
	which yields 
	$(\gamma_1-\gamma_2)(\varphi_{z,y}(\gamma_1)^2-\varphi_{z,y}(\gamma_2)^2)\leq
	0$ and \ref{prop:varphiz1} follows.
	
	\ref{prop:varphiz2}: It 
	follows 
	from the monotonicity of $A$ and \eqref{eq:contencionxzy} that, 
	for 
	every $w \in  Az +y$ and $\gamma \in \RPP$,
	
	\begin{align*}
		0 \leq 
		\Scal{\frac{z-x_{z,y}(\gamma)}{\gamma}-w}{x_{z,y}(\gamma)-z},
	\end{align*}
	which yields 
	\begin{align}\nonumber
		\frac{1}{\gamma}\|z-x_{z,y}(\gamma)\|^2 
		&\leq\scal{w}{z-x_{z,y}(\gamma)}\\
		&\leq \|{w}\| 
		\|z-x_{z,y}(\gamma)\|.\label{eq:desinf}
	\end{align}
	Thus, $\varphi_{z,y} (\gamma) \leq \|w\|$. Therefore, since 
	\cite[Proposition 20.36]{bauschkebook2017} implies that, for 
	every $z 
	\in \dom A$, $Az+y$ is nonempty, closed, and convex, 
	\cite[Theorem 
	11.10]{bauschkebook2017} yields
	\begin{align}
		(\forall \gamma \in \RPP)\quad \quad	
		\varphi_{z,y}(\gamma)\leq  
		\min_{
			w \in Az+y} \|{w}\|. \label{eq:desmin}
	\end{align}
	Hence, since $\varphi_{z,y} \geq 0$, \ref{prop:varphiz1} implies 
	that 
	$\lim_{\gamma \downarrow 0} 
	\varphi_{z,y}(\gamma)$ 
	exists. In turn, since $z \in \H\setminus \zer  
	(A+y)$, it follows from \eqref{eq:desmin} and 
	\eqref{varphiz=0} that
	\begin{equation} \label{eq:xtoz}
		0<\varphi_{z,y}(1) \leq \lim_{ \gamma \downarrow 0} 
		\varphi_{z,y}(\gamma)\leq  \min_{
			w \in Az+y} \|{w}\|, 
	\end{equation}
	which, in view of \eqref{prop:varphiz}, implies 
	\begin{equation}\label{eq:xzytoz}
		\lim_{\gamma \downarrow 0} x_{z,y}(\gamma)=z.
	\end{equation}
	Since $(\varphi_{z,y}(\gamma))_{(\gamma >0)}$ is bounded, by 
	\cite[Lemma 2.45]{bauschkebook2017}, there exists a sequence 
	$(\gamma_k)_{k \in \N} \subset \RPP$ and  $\overline{w} \in \H$  
	such that $\gamma_k  \downarrow 0$ and $\frac{z- 
		x_{z,y}(\gamma_k)}{\gamma_k} \weak \overline{w}$ as $k \to 
	+\infty$.  
	Therefore, \eqref{eq:contencionxzy}, \eqref{eq:xzytoz}, and 
	\cite[Proposition~20.38(i)]{bauschkebook2017} imply
	$\overline{w} 
	\in 
	A z +y $. Hence, noting that
	\begin{align} \nonumber
		(\varphi_{z,y}(\gamma_k))^2&= \left\|\frac{z- 
			x_{z,y}(\gamma_k)}{\gamma_k} - 
		\overline{w}\right\|^2+\|\overline{w}\|^2+2\Scal{\frac{z- 
				x_{z,y}(\gamma_k)}{\gamma_k}-\overline{w}}{\overline{w}}
		\\
		&\geq \|\overline{w}\|^2+2\Scal{\frac{z- 
				x_{z,y}(\gamma_k)}{\gamma_k}-\overline{w}}{\overline{w}},
		\label{eq:desparalimite}
	\end{align}
	we deduce
	$$
	\lim_{\gamma \downarrow 0}(\varphi_{z,y}(\gamma))^2 = 
	\lim_{k 
		\to 
		\pinf}(\varphi_{z,y}(\gamma_k))^2 \geq \|\overline{w}\|^2 
	\geq \min_{w \in Az+y} \|w\|^2.$$ Therefore, we obtain from 
	\eqref{eq:xtoz} that
	\begin{equation}
		\lim_{ \gamma \downarrow 0} \varphi_{z,y}(\gamma) =  \min_{w 
			\in 
			Az+y}\|w\|,
	\end{equation}
	and \ref{prop:varphiz2} follows.
	
	\ref{prop:varphiz23}: It follows from \eqref{eq:defsetC} and  
	\ref{prop:varphiz2} that $\dom A \subset\mathcal{C}$ . Let $z 
	\in \H \setminus \overline{\dom} A$ and let $\gamma >0$. The 
	firm nonexpansiveness of $J_{\gamma A}$ 
	\cite[Proposition~23.8(ii)]{bauschkebook2017} implies
	\begin{align}\label{eq:toproj} \nonumber
		\|x_{z,y}(\gamma)-P_{\overline{\dom} A} z \| &= \| J_{\gamma 
			A}(z-\gamma y) - J_{\gamma A}z + J_{\gamma A}z 
		-P_{\overline{\dom} 
			A} z\|\\ \nonumber
		& \leq \| J_{\gamma A}(z-\gamma y) - J_{\gamma A}z\| + 
		\|J_{\gamma A}z -P_{\overline{\dom} A} z\|\\
		& \leq \gamma \|y\| + \|J_{\gamma A}z 
		-P_{\overline{\dom} A} 
		z\|.
	\end{align}	
	Hence, by taking $\gamma 
	\downarrow 0$ in \eqref{eq:toproj} we conclude from  
	\cite[Theorem~23.48]{bauschkebook2017} that $ 
	x_{z,y}(\gamma) \to P_{\overline{\dom} A} z$ as $\gamma 
	\downarrow 
	0$. Then, by the continuity of the norm and $z \notin 
	\overline{\dom} A$, we deduce
	\begin{equation*}
		\lim_{\gamma \downarrow 0} \| z -x_{z,y}(\gamma) \| = \| z 
		-P_{\overline{\dom} A} z\|>  0.
	\end{equation*}
	Therefore, 
	$\varphi_{z,y}(\gamma)=\| z - x_{z,y}(\gamma) \| / \gamma \to + 
	\infty$ 
	as $\gamma 
	\downarrow 0$ and, hence, it follows from \eqref{eq:defsetC}
	that $z \in \H\setminus \mathcal{C}$. 
	
	\ref{prop:varphiz31}: If $z \in  \mathcal{C}$, it 
	follows from \ref{prop:varphiz1} that
	\begin{equation} \label{eq:xtoz2}
		0<\varphi_{z,y}(1) \leq \lim_{ \gamma \downarrow 0} 
		\varphi_{z,y}(\gamma)< +\infty. 
	\end{equation}
	Therefore 
	$\lim_{\gamma \downarrow 0} 
	x_{z,y}=z$ and the continuity of $B_3$ implies
	
	\begin{equation} \label{eq:b3to0}
		\lim_{\gamma  \downarrow 0}   B_3 x_{z,y} = B_3 z.
	\end{equation} 
	The result follows from \eqref{eq:xtoz2} and \eqref{eq:b3to0}.
	
	\ref{prop:varphiz32}: Set $B=B_1+B_2+B_3$ and let 
	$p=P_{\overline{\dom} A} z $. Since $B_3$ is locally bounded at 
	$p$, 
	there exists 
	$\delta_p \in \RPP$ such that $B_3(\mathcal{B}(p;\delta_p))$ is 
	bounded. Now since $y=Bz$ and
	\begin{equation}\label{eq:contencionJz}
		\frac{z-J_{\gamma A} z }{\gamma} \in A J_{\gamma A} z,
	\end{equation} it follows from \eqref{eq:contencionxzy} and 
	the 
	monotonicity of $A$ that
	\begin{align*} \nonumber
		0 \leq &  
		\Scal{\frac{z-x_{z,y}(\gamma)}{\gamma}-Bz-\frac{z-J_{\gamma 
					A}z}{\gamma}}{x_{z,y}(\gamma)-J_{\gamma A}z} \\ 
		=&  - \frac{1}{\gamma} \|x_{z,y}(\gamma)-J_{\gamma A} z\|^2 + 
		\scal{Bz}{J_{\gamma A}{z}-x_{z,y}(\gamma)}\\
		\leq& - \frac{1}{\gamma} \|x_{z,y}(\gamma)-J_{\gamma A} z\|^2 
		+ 
		\|Bz\|\|J_{\gamma A}{z} -x_{z,y}(\gamma)\|.
	\end{align*}
	Hence, we obtain 
	\begin{equation} \label{eq:cotaBz}
		\|x_{z,y}(\gamma)-J_{\gamma A}z\|\leq	\gamma \|Bz\|  .
	\end{equation}
	Additionally, by \cite[Theorem 23.48]{bauschkebook2017}, there 
	exists 
	$\gamma_1$ such that, for every $\gamma < \gamma_1$, $\| 
	J_{\gamma A} z-p\| \leq \delta_p/2$. By defining
	\begin{equation}
		\overline{\gamma}:=\begin{cases}
			\gamma_1, &\text{ if } Bz=0;\\
			\min \{ \delta_p/(2\|Bz\|),\gamma_1\}, &\text{ if } 
			Bz\neq 0,
		\end{cases}
	\end{equation}
	it follows from 
	\eqref{eq:cotaBz} that, for every $\gamma < 
	\overline{\gamma}$,
	\begin{align*}
		\|x_{z,y}(\gamma) - p \| &\leq \| x_{z,y}(\gamma) -J_{\gamma 
			A} z\| 
		+ \|J_{\gamma A} z - p \| \\
		& \leq \gamma \|Bz\| + \frac{\delta_p}{2}\\
		&< \delta_p,
	\end{align*}
	which yields $(x_{z,y}(\gamma))_{0<\gamma\leq 
		\overline{\gamma}} 
	\subset{\mathcal{B}(p,\delta_p)}$ and, thus, $(\|B_3 
	z-B_3x_{z,y}(\gamma)\| 
	)_{0<\gamma\leq \overline{\gamma}}$ is bounded. Therefore, 
	since $z\in \H 	\setminus 	\mathcal{C}$ implies 
	$\varphi_{z,\gamma} (\gamma) \to + \infty$ as $\gamma 
	\downarrow 0$,  the result follows. 
	
	\ref{prop:varphiz33}: Since $p=P_{\overline{\dom} A} z \notin 
	\bdry \dom B_3$,  $B_3$ is 	locally bounded at $p$ 
	\cite[Theorem~21.18]{bauschkebook2017} and 
	the result follows from 
	\ref{prop:varphiz32}.\qed
\end{proof}

\begin{remark} \begin{enumerate}
		\item In the case $B_2=0$, by setting $y=(B_1+B_3)z$ in 
		Lemma~\ref{prop:varphiz}\eqref{prop:varphiz1}\&\eqref{prop:varphiz2},
		we recover \cite[Lemma 2.2(1)]{BricenoDavis2018}.
		\item Realizing that \cite[Lemma 2.2(2)]{BricenoDavis2018} is 
		valid for every $z \in \dom A$, it is a particular case of 
		Lemma~\ref{prop:varphiz}\eqref{prop:varphiz23}\&\eqref{prop:varphiz31}.
	\end{enumerate}
\end{remark}
Now we state our main result.
\begin{teo}\label{teo:convergencia}
	In the context of Problem~\ref{prob:problem1},
	suppose that one of the following holds:
	\begin{enumerate}
		\item \label{teo:hyp1} $X \subset \dom A$.
		
		\item \label{teo:hyp2}  $\overline{\dom} A \subset \dom B_3$ 
		and 
		$B_3$ is locally bounded in $\dom B_3$.
		
		\item \label{teo:hyp3}  $\overline{\dom} A \subset \inte \dom 
		B_3$.
		
	\end{enumerate} Let $ \varepsilon 
	\in \left]0, 1\right[$, set $\rho= \min\{2 \beta \varepsilon, 
	\sqrt{1-\varepsilon}/L\}$, let $\sigma \in \left]0,1\right[$, let 
	$\theta 
	\in \left]0,\sqrt{1-\varepsilon} -L\rho\sigma\right[$, let $z_0 
	\in \dom B_3$, and 
	consider the sequence $(z_n)_{n\in \N}$ defined by the 
	recurrence
	\begin{equation}\label{eq:algorithm}
		(\forall n\in\N)\quad 
		\begin{array}{l}
			\left\lfloor
			\begin{array}{l}
				x_n= J_{\gamma_n A}\big(z_n-\gamma_n 
				(B_1+B_2+B_3)z_n\big)\\
				z_{n+1}= P_X\big(x_n + \gamma_n 
				(B_2+B_3)z_n-\gamma_n 
				(B_2+B_3)x_n\big),
			\end{array}
			\right.
		\end{array}
	\end{equation}
	where, for every $n \in \N$, $\gamma_n$ is the 
	largest $\gamma \in \{ \rho \sigma,  \rho \sigma^2,  
	\rho \sigma^3, \cdots \}$ satisfying
	\begin{equation}\label{eq:thetateo}
		\gamma \big\|B_3z_n-B_3 J_{\gamma A}\big(z_n - \gamma 
		(B_1+B_2+B_3)z_n\big)\big\| 
		\leq \theta\big\|z_n-J_{\gamma A} \big(z_n - \gamma 
		(B_1+B_2+B_3)z_n\big)\big\|.
	\end{equation} 
	Moreover, assume that at least one of the following additional 
	statements hold:
	\begin{enumerate}[label=(\roman*)]
		\item \label{teo:hypa} $\displaystyle \liminf_{n \to \infty} 
		\gamma_n = 
		\delta >0$.
		\item \label{teo:hypb} $B_3$ is uniformly continuous in any 
		weakly 
		compact subset of $\overline{\conv} (\dom A \cup X)$.
	\end{enumerate}
	Then, $(z_n)_{n \in \N}$ converges weakly to a solution to
	Problem~\ref{prob:problem1}.
	
\end{teo}

\begin{proof}
	
	Set $B=B_1+B_2+B_3$ and fix $n \in \N$. 
	If $z_n \in \mathcal{C}$, where $\mathcal{C}$ is defined in 
	\eqref{eq:defsetC}, then $\gamma_n$ is well defined in view of 
	Lemma~\ref{prop:varphiz}\eqref{prop:varphiz31}. In particular, if 
	\ref{teo:hyp1} holds, $\gamma_n$ is well defined in view of 
	Lemma~\ref{prop:varphiz}\eqref{prop:varphiz23}. Now suppose 
	that 
	$z_n 
	\in \H \setminus \mathcal{C}$. If $n=0$, it is clear that 
	$z_0 
	\in \dom B_3 \setminus \mathcal{C}$. Otherwise, since $X 
	\subset 
	\dom B_3$, we have 
	$z_n 
	\in \dom B_3 \setminus \mathcal{C}$. Now, if we assume 
	\ref{teo:hyp2}, then $B_3$ is locally bounded in 
	$P_{\overline{\dom} 
		A}z_n\in\overline{\dom} A$ and $\gamma_n$ is well defined 
	from 
	Lemma~\ref{prop:varphiz}\eqref{prop:varphiz32}. Similarly, if we 
	assume \ref{teo:hyp3}, $\gamma_n$ is well defined from 
	Lemma~\ref{prop:varphiz}\eqref{prop:varphiz33}.
	
	Now, let $z^* \in \zer(A+B)\cap X$. Note that the maximal 
	monotonicity of 
	$A+B_3$, the full domain of $B_1$ and $B_2$, and \cite[Corollary 
	25.5(i)]{bauschkebook2017} imply that $A+B_2+B_3$ and 
	$A+B$ 
	are 
	maximally 
	monotone. Then, since 
	$B_2+B_3$ is continuous and single valued in $\dom 
	(B_2+B_3)=\dom 
	B_3 \supset \dom A \cup X$ and $B_1$ is $\beta$-cocoercive, it 
	follows from \cite[Proposition 
	2.1(1)\&(2)]{BricenoDavis2018} that, for every $n \in \N$, 
	we have
	\begin{align}\label{eq:LemmaDavis}\nonumber
		\|z_{n+1}- z^*\|^2 \leq & \|z_{n}- z^*\|^2-(1-\varepsilon) 
		\|z_n-x_n\|^2+\gamma_n^2\|(B_2+B_3)z_n-(B_2+B_3)x_n\|^2\\
		& \hspace*{5cm} -\frac{\gamma_n}{\varepsilon}(2\beta 
		\varepsilon 
		- \gamma_n)\|B_1z_n-B_1z^*\|^2.
	\end{align}
	Note that the Lipschitz property of $B_2$ and 
	\eqref{eq:thetateo} yield
	\begin{align}\label{eq:LemmaDavis2}
		\gamma_n^2\|(B_2+B_3)z_n-(B_2+B_3)x_n\|^2  & \leq 
		(L\gamma_n\|z_n-x_n\|+\gamma_n\|B_3z_n-B_3x_n\|)^2  
		\nonumber \\
		& \leq (L\gamma_n+\theta)^2\|z_n-x_n\|^2  \nonumber\\
		& \leq  (L\rho \sigma +\theta)^2\|z_n-x_n\|^2.
	\end{align}
	Hence, it follows from \eqref{eq:LemmaDavis} and 
	\eqref{eq:LemmaDavis2} that
	\begin{align*}
		(\forall n \in \N)\quad	\|z_{n+1}- z^*\|^2
		\leq & \|z_n-z^*\|^2-\big((1-\varepsilon)-(L\rho \sigma
		+\theta)^2\big)\|z_n-x_n\|^2.
	\end{align*}
	Therefore, since $(1-\varepsilon-(L \rho \sigma +\theta)^2)>0$, 
	\cite[Lemma~3.1(i)]{Comb2001} implies that 
	$(\|z_n-z^*\|)_{n \in \N}$ is a convergent sequence and 
	\begin{equation}
		\label{e:zxto0}
		z_n-x_n \to 0.
	\end{equation} 
	Let $z \in \H$ be a weak limit point of the 
	subsequence 
	$(z_n)_{n\in K}$ for some $K \subset \N$. Then $z$ is also a 
	weak 
	limit point of 
	$(x_n)_{n\in K}$ in view of \eqref{e:zxto0}. Since $X$ is closed 
	and convex, 
	and 
	$(z_n)_{n\in K}$ is a sequence in $X$, we conclude that $z 
	\in X$. Let us prove that $z \in \zer (A+B)$.
	
	\ref{teo:hypa}: Assume that $\liminf_{n\to 
		+\infty}\gamma_n=\delta 
	>0$. 
	Then, there exists $n_0 \in \N$ such that $\inf_{n \geq n_0} 
	\gamma_n 
	\geq \delta$. Hence, \eqref{eq:thetateo}, \eqref{eq:algorithm}, 
	the 
	Lipschitz 
	continuity of $B_2$, and the 
	cocoercivity of $B_1$ yield
	
	\begin{align}\nonumber
		(\forall n \geq  n_0)	\quad \|B z_n- B x_n\| &\leq \|B_1 
		z_n - B_1 
		x_n\| + 
		\|B_2 z_n - B_2 x_n\| 
		+ \|B_3 z_n - B_3x_n\|\\\label{eq:con_z-x}
		&\leq \left(\frac{1}{\beta}+L+\frac{\theta}{\delta}\right) 
		\|z_n-x_n\|,
	\end{align} 
	which implies $B z_n -B x_n \to 0$ in view of \eqref{e:zxto0}. 
	Hence, 
	it follows from  
	\eqref{eq:algorithm} that
	\begin{equation} \label{eq:contencion1}
		(\forall n \in \N) \quad 
		u_n:=\frac{z_{n}-x_n}{\gamma_n}-Bz_n+Bx_n 
		\in 
		(A+B)x_n,
	\end{equation}
	and \eqref{e:zxto0} and $\liminf_{n\to +\infty}\gamma_n=\delta 
	>0$ 
	imply that $u_n \to 0$. 
	Therefore, since $x_n \weak z$, $n \in K$, 
	the weak-strong closure of the graph of the maximally monotone 
	operator $A+B$ and 
	\eqref{eq:contencion1} yield $z \in \zer(A+B)$. The 
	convergence follows from 
	\cite[Lemma~2.47]{bauschkebook2017}.

	\ref{teo:hypb}: Without loss of generality, suppose
	that 
	$\lim_{n \to \infty, n \in K } \gamma_n=0$. Our choice of 
	$\gamma_n$ 
	guarantee that, for every 
	$n 
	\in 
	K$, we have
	\begin{equation} \label{eq:destheta3}
		\widetilde{\gamma}_n \|B_3z_n-B_3 J_{\widetilde{\gamma}_n A} 
		(z_n 
		- 
		\widetilde{\gamma}_n Bz_n)\| >  
		\theta\|z_n-J_{\widetilde{\gamma}_n A} (z_n 
		- \widetilde{\gamma}_n Bz_n)\|,
	\end{equation}
	where 
	$\widetilde{\gamma}_n:=\frac{\gamma_n}{\s}>\gamma_n$. 
	Now, by the nonincreasing property of $\gamma \mapsto 
	\frac{1}{\gamma}\|z-J_{\gamma A} (z - \gamma Bz)\|$ provided 
	by 
	Lemma~\ref{prop:varphiz}\eqref{prop:varphiz1} with $y = 
	B 
	z$, we have
	\begin{align}\label{eq:nonincreasinggamma}
		\frac{1}{\widetilde{\gamma}_n}
		\|z_n-J_{\widetilde{\gamma}_n A} 
		(z_n - 	\widetilde{\gamma}_n Bz_n)\| \leq 
		\frac{1}{\gamma_n}\|z_n-J_{\gamma_n A} (z_n - \gamma_n 
		Bz_n)\|, 
	\end{align}
	which is equivalent to 
	\begin{equation}
		\|z_n-J_{\widetilde{\gamma}_n A} (z_n - \widetilde{\gamma}_n 
		Bz_n)\| \leq 
		\frac{1}{\s}\|z_n-x_n\|.
	\end{equation}
	Thus, by defining 
	\begin{equation}\label{eq:xtilde}
		(\forall n \in K) \quad 
		\widetilde{x}_n=J_{\widetilde{\gamma}_n	A} 
		(z_n-\widetilde{\gamma}_nBz_n),
	\end{equation} 
	\eqref{e:zxto0} implies that 
	\begin{equation}
		\label{e:auxseqtoz}
		z_n-\widetilde{x}_n \to 0\quad\text{as}\quad n \to\infty, n 
		\in 
		K.
	\end{equation}
	Therefore, since $z_n \weak z$, $n \in K$, we have 
	$\widetilde{x}_n \weak z$ as $n \to\infty$, $n \in K$. 
	Furthermore, 
	\eqref{eq:xtilde} yields
	\begin{equation}\label{eq:contencion2}
		\frac{z_n-\widetilde{x}_n}{\widetilde{\gamma}_n}+ 
		B\widetilde{x}_n-Bz_n \in 
		(A+B)\widetilde{x}_n.
	\end{equation}
	Since $\{z\} \cup \bigcup_{n \in K} [\widetilde{x}_n,z_n]$ is a 
	weakly 
	compact subset of $ \overline{\conv} (\dom A \cup X)$ 
	\cite[Lemma 
	3.2]{salzo2017}, it follows from 
	the uniform continuity of $B_3$ and \eqref{e:auxseqtoz} that
	\begin{equation}\label{eq:b3tozero}
		B_3\widetilde{x}_n-B_3z_n \to 0 \text{ as } 
		n \to \infty,\ n \in K,
	\end{equation}
	which, combined with \eqref{eq:destheta3}, yields
	$(z_n-\widetilde{x}_n)/\widetilde{\gamma}_n \to 0\text{ as 
	} n \to \infty,\ n \in K$. Moreover, 
	the Lipschitz continuity of $B_1+B_2$, \eqref{e:auxseqtoz}, and 
	\eqref{eq:b3tozero}
	imply 
	\begin{equation}
		B\widetilde{x}_n-Bz_n \to 0 \text{ as } n \to \infty,\ n \in 
		K.
	\end{equation}
	Altogether, the convergence follows, as in the case 
	\ref{teo:hypa}, 
	from \eqref{eq:contencion2}, the weak-strong closedness of the 
	graph 
	of
	the maximally 
	monotone operator $A+B$, and 
	\cite[Lemma~2.47]{bauschkebook2017}.\qed
\end{proof}
\begin{remark}\label{rem:pct}
	\begin{enumerate}
		\item In Theorem~\ref{teo:convergencia}, if $B_3=0$, we have 
		$\dom 
		B_3= \H$ and, for all $n 
		\in \N$, $\gamma_n=\sigma \rho = \sigma 
		\min\{2\beta\varepsilon,\sqrt{1-\varepsilon}/L\}$. Since, in 
		this 
		case $(\gamma_n)_{n\in\N}$ is constant, the largest step-size 
		is 
		obtained by taking $\varepsilon=\varepsilon(L,\beta) 
		=2/(1+\sqrt{1+16\beta^2L^2})$, which satisfies 
		$2\beta\varepsilon = 
		\sqrt{1-{\varepsilon}}/L = \chi(L,\beta)$, where
		\begin{equation}
			\chi(L,\beta)= \frac{4\beta}{1+\sqrt{1+16\beta^2L^2}},
		\end{equation} 
		and $\gamma_n \equiv \gamma=  \sigma \chi(L,\beta) 
		\in \left]0, \chi(L,\beta)
		\right[$. Hence, 
		we recover the result in 
		\cite[Theorem~2.3(1)]{BricenoDavis2018} for constant 
		stepsizes. 
		Additionally, if $B_2=0$ and $X=\H$, we have 
		$\varepsilon(L,\beta) 
		\to 1$ and $\chi(L,\beta)\to 2\beta$ as $ L \to 0$ and 
		$\gamma_n \equiv 2\beta \sigma \in \left]0, 2\beta\right[$, 
		recovering the the 
		forward 
		backward algorithm \cite{lionsandmercier1979}.
		On the other hand, if $B_1=0$, we have $\chi(L,\beta) \to 
		1/L$ 
		as $\beta \to \infty$ and
		$\gamma_n \equiv \sigma/L \in \left]0,1/L\right[$, recovering 
		the 
		result in 
		\cite{Tseng2000SIAM} for constant step-sizes.
		\item Suppose that $B_2=0$ and $X\subset \dom A$. Then by 
		taking $L 
		\to 0$, we have $\rho \to 2\beta \varepsilon$ and $\theta \in 
		\left]0,\sqrt{1-\varepsilon}\right[$. Hence, 
		Theorem~\ref{teo:convergencia} 
		recovers \cite[Theorem~2.3(2)]{BricenoDavis2018} noting that 
		the 
		uniform continuity in weakly compact subsets of 
		$\overline{\conv}(\dom A \cup X)=\overline{\dom} A$ is needed.
		We 
		hence 
		generalize 
		\cite[Theorem~2.3(1)\&(2)]{BricenoDavis2018} to the case when 
		$X\not\subset  \dom A$.
	\end{enumerate}
\end{remark}

\section{Application to Convex Optimization with Nonlinear 
	Constraints}\label{sec:4}

In this section we consider the following optimization problem.
\begin{pro}\label{prob:opti}
	Let $f\in\Gamma_0 (\H)$, let $g\in \Gamma_0(\G)$, let $h:\H \to 
	\R$ 
	be a convex G\^ateaux differentiable
	function such that $\nabla h$ is $\beta^{-1}$-Lipschitzian for 
	some $\beta \in \RPP$, let $M \colon \H \to \G$ be a bounded 
	linear 
	operator, and let $e \colon \H \to \RX^p \colon x 
	\mapsto 
	(e_i(x))_{1\le i 
		\le p}$ be such that, for every $i\in \{1,\ldots,p\}$, $e_i$ 
	is 
	convex and 
	G\^ateaux differentiable in $\inte \dom e_i$, $\dom e_i$ is 
	closed, 
	$\cap_{i=1}^p 
	\interior \dom 
	e_{i}\neq \varnothing$, and $\dom \partial f \subset  
	\cap_{i=1}^n \interior \dom  e_i $. Assume that $0 \in 
	\sri(\dom g - M (\dom f))$ and that 
	\begin{equation}\label{e:slater}
		\begin{cases}
			(\forall i \in \{1,\ldots,p\}) \quad \quad \lev_{\leq 0} 
			e_i \subset 
			\inte \dom e_i;\\
			\dom (f + g \circ M) \cap \bigcap_{i=1}^n \lev_{<0} e_i 
			\neq \varnothing.
		\end{cases}
	\end{equation} The problem is to 
	\begin{equation}\label{eq:probopti}
		\min_{e(x)\in \RM^p} f(x) + g (Mx) + h(x),
	\end{equation}
	and we assume that solutions exist.
\end{pro}
It follows from  \eqref{e:slater} 
and \cite[Proposition~27.21]{bauschkebook2017} that $\hat{x} \in 
\H$ is a solution to 
Problem~\ref{prob:opti} if and only if there exists $\hat{v} \in 
\RP^p$ such that
\begin{equation}
	-\sum_{i=1}^p \hat{v}_i\nabla e_i(\hat{x}) \in \partial 
	(f+g\circ M + h) \quad
	\text{ and } \quad	(\forall i \in \{1,\ldots,p\}) \quad  
	\begin{cases}
		e_i(\hat{x}) \leq 0,\\
		\hat{v}_i e_i(\hat{x}) =0.
	\end{cases}
\end{equation}
Hence, by 
\cite[Example~16.13\,\&\,Example~6.42(i)]{bauschkebook2017}, 
we deduce that
	\begin{equation}\label{eq:inclulag}
		\begin{cases}
			0 \in \partial ( f + g \circ M + h) (\hat{x}) + 
			\sum_{i=1}^p 
			\hat{v}_i  
			\nabla e_i (\hat{x}) \\
			0 \in  N_{\RP^p} (\hat{v}) - e(\hat{x}).
		\end{cases}
	\end{equation}
	Then, we deduce from \cite[Theorem 
	16.47]{bauschkebook2017}
	that there exists $\hat{u} 
	\in \G$ such that 
	\eqref{eq:inclulag} reduces to 
	\begin{equation*}
		\begin{cases}
			0\in \partial  f (\hat{x}) + M^*\hat{u} + \nabla h 
			(\hat{x})  + 
			\sum_{i=1}^p \hat{v}_i \nabla e_i (\hat{x})   \\
			0  \in \partial g^*(\hat{u})-M\hat{x}\\
			0 \in  N_{\RP^p} (\hat{v}) - e(\hat{x}),
		\end{cases}
	\end{equation*}
	which is equivalent to 
	\begin{equation}\label{eq:incluAB}	
		\begin{pmatrix}
			0\\
			0\\
			0
		\end{pmatrix} \in \begin{pmatrix}
			\partial f (\hat{x})\\
			\partial g^*(\hat{u})\\
			N_{\RP^p} (\hat{v} )
		\end{pmatrix}
		+\begin{pmatrix}
			\nabla h (\hat{x})\\
			0\\
			0
		\end{pmatrix}+
		\begin{pmatrix}
			M^* \hat{u} \\
			-M \hat{x}\\
			0
		\end{pmatrix}
		+ \begin{pmatrix}
			\sum_{i=1}^p\hat{v}_i \nabla e_i (\hat{x})  \\\
			0\\
			-e(\hat{x})
		\end{pmatrix}.
	\end{equation}
	
	\begin{prop}\label{prop:B3cont}
		In the context of Problem~\ref{prob:opti}, let $X=X_1\times 
		X_2 \times 
		X_3  \subset \dom \partial f \times \dom \partial g^* \times 
		\RP^p$ be 
		nonempty, closed, and convex, and define the operator 
		\begin{align}\label{eq:defB3prop} \displaystyle
			B_3 \colon \H \times \G \times \R &\to 2^{\H \times \G 
				\times \R}
			\nonumber\\
			(x,u,v) &\mapsto \begin{cases}
				\!\!\left\{\!\left(\displaystyle{\sum_{i=1}^p} v_i 
				\nabla  e_i (x) , 
				0, 
				-e(x)\right)\!\!\right\},\hspace{-.3cm}& \text{ if } 
				v \in 
				\RP^p \text{ and } x \in \bigcap_{i=1}^p \interior 
				\dom 
				e_{i};\\
				\varnothing,& \text{ otherwise.} \\				
			\end{cases}
		\end{align} 
		Then, the following hold:
		\begin{enumerate}
			\item \label{prop:B3cont1} $B_3$ is maximally monotone.
			
			\item \label{prop:B3cont2} Suppose that one of the 
			following 
			holds:
			
			\begin{enumerate}
				\item\label{prop:hip1} $(\nabla e_i)_{1\leq i \leq 
					p}$ are 
				bounded and uniformly 
				continuous in every weakly compact subset of 
				$\overline{\dom} \partial f$. 
				\item\label{prop:hip2} $\H$ is finite dimensional 
				and  
				$(\nabla e_i)_{1\leq i \leq p}$ are 
				continuous in every compact subset of 
				$\overline{\dom} \partial 
				f$.
			\end{enumerate}
			Then,
			$B_3$ is uniformly continuous in every compact subset of 
			$\overline{\dom} \partial f \times \overline{\dom} 
			\partial 
			g^* 
			\times 
			\RP^p$. 
		\end{enumerate}
		
	\end{prop}
	
	\begin{proof} \ref{prop:B3cont1}: Consider the saddle-function 
		\begin{align*}
			\ell \colon \H \times \G \times  \R^p & \to \RXX\\
			(x,u,v) &\mapsto\begin{cases} e(x)\cdot v, & \text{ 
					if } v \in 	\RP^p \text{ and } x \in 
				\bigcap_{i=1}^p 
				\dom e_{i};\\
				+\infty, &\text{ if } v \in 	\RP^p \text{ and } x 
				\notin 
				\bigcap_{i=1}^p \dom 
				e_{i};		\\ 
				-\infty, &\text{ if } v \notin \RP^p.
			\end{cases}
		\end{align*} 
		Note that, if $v \in \RP^p$, 
		\begin{equation}
			\ell\colon(x,u,v) \mapsto \begin{cases} 
				e(x)\cdot v, 
				& 
				\text{ 
					if } x \in \bigcap_{i=1}^p 
				\dom e_{i};\\
				+\infty, &\text{ otherwise}
			\end{cases}
		\end{equation}
		and, if  $v \notin \RP^p$, 
		$\ell(\cdot,\cdot,v)\equiv-\infty$. Hence, 
		for every $v 
		\in \R^p$, $\ell(\cdot, \cdot , v)$ is lower-semicontinuous. 
		Additionally, 
		if $(x,u) \in \bigcap_{i=1}^p \dom e_i\times\G$, we have
		\begin{equation}
			-\ell\colon(x,u,v)\mapsto \begin{cases} 
				-e(x)\cdot v, 
				& 
				\text{ 
					if } v \in 	\RP^p ;\\
				+\infty, &\text{ otherwise}
			\end{cases}
		\end{equation}
		and, if $x \notin \bigcap_{i=1}^p \dom e_i$ and $u\in\G$,
		\begin{equation}
			-\ell\colon(x,u,v)\mapsto\begin{cases} 
				-\infty & 
				\text{ if } v \in 	\RP^p ;\\
				+\infty, &\text{ otherwise}.
			\end{cases}
		\end{equation}
		Therefore, for every $(x,u) \in \H\times \G$, 
		$-\ell(x,u,\cdot)$ is lower-semicontinuous. Furthermore,
		\begin{equation}
			(\forall (x,u) \in \H\times \G) (\forall v \in \R^p) 
			\quad B_3(x,u,v) = 
			\partial \ell(\cdot,\cdot,v)(x,u) \times  \partial 
			(-\ell(x,u,\cdot))(v).
		\end{equation} The result 
		follows from \cite[Corollary 1]{Rockafellar70}.
		
		\ref{prop:B3cont2}: First, assume \ref{prop:hip1}.
		Let $Y=Y_1\times Y_2 \times Y_3 \subset \overline{\dom} 
		\partial f \times  \overline{\dom}  \partial g^* \times 
		\RP^p$ be a weakly compact set. Let 
		$\boldsymbol{x}= (x_1,u_1,v_1)$ and 
		$\boldsymbol{y}=(x_2,u_2,v_2)$ in $Y$, fix $i \in 
		\{1\ldots,p\}$, define 
		$\rho_i \colon \left[0,1\right] \to \R \colon t 
		\mapsto e_i(x_1+t(x_2-x_1))$,
		which is differentiable in $\left]0,1\right[$. Since $Y_1$ is 
		weakly compact, by 
		\cite[Theorem 3.37]{bauschkebook2017}, $\conv Y_1$ is also 
		weakly 
		compact. Moreover, we deduce from the boundedness of  
		$\nabla e_i$ in $\conv Y_1 
		\subset 
		\overline{\dom} \partial f$ that there 
		exists $K_i>0$ such that $\sup_{x \in \conv Y_1}\|\nabla 
		e_i(x)\| \leq 
		K_i$. Therefore, since ${\rho_i}'\colon t \mapsto 
		\scal{\nabla 
			e_i(x_1+t(x_2-x_1))}{x_2-x_1}$, we obtain
		\begin{align*}
			|e_i(x_2)-e_i(x_1)| &=|\rho_i(1)-\rho_i(0)|\\
			&= \left|\int_0^1 {\rho_i}'(t)dt\right|\\
			&= \left|\int_0^1\scal{\nabla 
				e_i(x_1+t(x_2-x_1))}{x_2-x_1} dt\right|\\
			&\leq \int_0^1 \|\nabla 
			e_i(x_1+t(x_2-x_1))\|\|x_2-x_1\|dt\\
			&\leq  K_i \|x_2-x_1\|.
		\end{align*} 
		Thus,  we conclude $|e_i(x_1)-e_i(x_2)| \leq K_i \|x_2-x_1\|$ 
		and therefore 
		\begin{equation}\label{eq:eKlip}
			\|e(x_2)-e(x_1)\| \leq K \|x_2-x_1\|,
		\end{equation}
		where $K=\max_{i\in \{1,\ldots,p\}} K_i$.  Since $Y_3$ is 
		weakly compact, there 
		exists $V >0$ such that $\sup_{v \in Y_3} \|v\| \leq V$
		\cite[Lemma~2.36]{bauschkebook2017}. 
		Let $\varepsilon > 0 $. The uniform continuity of $(\nabla 
		e_i)_{1\leq i 
			\leq p}$ implies 
		the existence of $\delta >0$ such that
		\begin{equation}\label{eq:geaco}
			(\forall i \in \{1,\ldots,p\}) (\forall (z_1,z_2) \in 
			Y_1^2)\quad 
			\|z_1-z_2\|< \delta  \Rightarrow \| 
			\nabla e_i (z_1) - \nabla e_i ( z_2)  \|^2 \leq 
			\frac{\varepsilon^2}{4 p
				V^2}.
		\end{equation}
		Now suppose that 
		\begin{equation*}
			\| \boldsymbol{x} - \boldsymbol{y}  \|^2 \leq \min 
			\left\{  
			\frac{\varepsilon^2}{4p  K^2}   , \delta^2  \right\}.
		\end{equation*}
		Then, \eqref{eq:eKlip}, the convexity of $\|\cdot\|^2$, and 
		\eqref{eq:geaco} imply
		\begin{align*}
			\|B_3 &\boldsymbol{x} - B_3 \boldsymbol{y}\|^2 = \left\| 
			\sum_{i=1}^p 
			v_{1,i} \nabla e_i(x_1) - v_{2,i} \nabla e_i(x_2) 
			\right\|^2 + 
			\|e(x_1)-e(x_2) \|^2\\
			&\leq 2  \left\|\sum_{i=1}^p (v_{1,i} -v_{2,i}) \nabla 
			e_i(x_1)  
			\right\|^2 + 
			2  \left\|\sum_{i=1}^p v_{2,i} (\nabla e_i(x_1)-\nabla 
			e_i(x_2))  
			\right\|^2  +  K^2 \|x_1-x_2\|^2 \\
			&\leq 2p  \sum_{i=1}^p |v_{1,i} -v_{2,i}|^2 \left\|\nabla 
			e_i(x_1)  
			\right\|^2 + 
			2p   \sum_{i=1}^p |v_{2,i}|^2 \left\|\nabla 
			e_i(x_1)-\nabla 
			e_i(x_2)  
			\right\|^2  +  K^2 \|x_1-x_2\|^2 \\
			&\leq   2p  K^2  \left\| v_{1} - v_{2}\right\|^2 + 2p 
			V^2  
			\sum_{i=1}^p \left\| 
			\nabla e_i (x_1) -  \nabla e_i (x_2)\right\|^2    +  
			K^2\|x_1-x_2\|^2\\
			&\leq   2p K^2 \left\| \boldsymbol{x} - \boldsymbol{y} 
			\right\|^2 
			+\frac{\varepsilon^2}{2}\\
			& \leq \varepsilon^2.
		\end{align*}
		Therefore $B_3$ is uniformly continuous in $Y$. 
		
		Now, assume \ref{prop:hip2}. Since $\H$ is finite dimensional 
		the 
		weak and strong topologies coincide \cite[Fact 
		2.33]{bauschkebook2017}. Hence, since  $\nabla e_i$ is 
		continuous, it 
		is bounded and uniformly continuous in every compact subset 
		of $X$. 
		The result follows from \ref{prop:hip1}.\qed
	\end{proof}
	\begin{remark}
		Note that if, for every $i \in \{1,\ldots,p\}$, $\nabla e_i$ 
		is bounded 
		and uniformly continuous in every weakly compact subset of 
		$\dom f$, since 
		$\overline{\dom}\partial f \subset \dom 
		f$, $B_3$ is uniformly continuous 
		in every compact 
		subset of $\overline{\dom} \partial f \times \overline{\dom}
		\partial g^* \times \RP^p$ in view of 
		Proposition~\ref{prop:B3cont}.
	\end{remark}
	\begin{prop} In Problem~\ref{prob:opti}, let 
		$X=X_1\times X_2 \times X_3  \subset 
		\dom 
		\partial f \times \dom \partial g^* \times \RP^p$ be 
		nonempty, 
		closed, 
		and convex, let $ \varepsilon 
		\in\left]0, 1\right[$, set $\rho= \min\{2 \beta \varepsilon, 
		\sqrt{1-\varepsilon}/\|M\|\}$, let $\sigma \in 
		\left]0,1\right[$, let 
		$\theta 
		\in\left]0,\sqrt{1-\varepsilon} -\|M\|\rho\sigma\right[$.
		For every $\boldsymbol{z}=(z^1,z^2,z^3) \in \H \times \G 
		\times \R^p$ 
		define 
		$\boldsymbol{\Phi}_{\boldsymbol{z}}\colon \gamma \mapsto 
		(\Phi_{\boldsymbol{z}}^1 
		(\gamma),\Phi_{\boldsymbol{z}}^2(\gamma),\Phi_{\boldsymbol{z}}^3(\gamma))$,
		where
		\begin{align}\label{eq:defx123}\nonumber
			\Phi_{\boldsymbol{z}}^1\colon \gamma&\mapsto 
			\prox_{\gamma 
				f}\left(z^1-\gamma \left( \nabla h (z^1) + M^*z^2 + 
			\sum_{i=1}^p 
			z_i^3 \nabla e_i (z^1)\right)\right)\\ \nonumber
			\Phi_{\boldsymbol{z}}^2\colon 
			\gamma&\mapsto\prox_{\gamma
				g^*} 
			(z^2+\gamma 
			M z^1)\\ 
			\Phi_{\boldsymbol{z}}^3\colon \gamma&\mapsto 
			P_{\RP^p}\big(z^3+\gamma e 
			(z^1)\big).
		\end{align}
		Let $\boldsymbol{z}_0=(z_0^1,z_0^2,z_0^3)\in \H\times \G 
		\times 
		\R^p$ and consider the 	recurrence, for every $n\in\N$,
		\begin{equation}\label{eq:algorithmop}
			\begin{array}{l}
				\left\lfloor
				\begin{array}{l}
					x_n^1=\Phi_{\boldsymbol{z}_n}^1(\gamma_n)\\
					x_n^2=\Phi_{\boldsymbol{z}_n}^2(\gamma_n) \\
					x_n^3=\Phi_{\boldsymbol{z}_n}^3(\gamma_n)\\
					z_{n+1}^1= P_{X_1}\big(x_n^1 + \gamma_n (M^* 
					z_n^2 + 
					\sum_{i=1}^p z_{n,i}^3 \nabla e_i(z_n^1) 
					)-\gamma_n (M^* 
					x_n^2 + \sum_{i=1}^p  x_{n,i}^3 \nabla e_i(x_n^1) 
					)\big)\\
					z_{n+1}^2= P_{X_2} (x_n^2 - \gamma_n M z_n^1 + 
					\gamma_n 
					M x_n^1) \\
					z_{n+1}^3= P_{X_3} (x_n^3 - \gamma_n e(z_n^1) + 
					\gamma_n e  (x_n^1))\\
					\boldsymbol{z}_{n+1}=(z_{n+1}^1,z_{n+1}^2,z_{n+1}^3),
				\end{array}
				\right.
			\end{array}
		\end{equation}
		where $\gamma_n$ is the largest $\gamma \in \{ \rho\sigma, 
		\rho\sigma^2, \rho\sigma^3,  \dots \}$ satisfying
		\begin{multline}\label{eq:thetateoop}
			\gamma^2 \left( \left\|\sum_{i=1}^p  z_{n,i}^3\nabla e_i 
			(z_n^1)  - 
			\Phi_{\boldsymbol{z}_n,i}^3(\gamma) \nabla e 
			\big(\Phi_{\boldsymbol{z}_n}^1(\gamma)\big)\right\|^2+ 
			\left\| 
			e(z_{n}^1) - 
			e\big(\Phi_{\boldsymbol{z}_n}^1(\gamma)\big) 
			\right\|^2\right)\\
			\leq \theta^2  
			\pnorm{\boldsymbol{z}_n - 
				\boldsymbol{\Phi}_{\boldsymbol{z}_n}(\gamma)}^2.
		\end{multline}
		Moreover, assume that at least one of the following 
		additional 
		statements hold:
		\begin{enumerate}[label=(\roman*)]
			\item \label{teo:ophypa} $\displaystyle \liminf_{n \to 
				\infty} 
			\gamma_n = \delta >0$.
			\item \label{teo:ophypb} For every $i \in \{1, 
			\ldots,p\}$, $\nabla e_i$ 
			is bounded and uniformly continuous in every weakly 
			compact 
			subset of $\overline{\dom}\partial f$.  
		\end{enumerate}
		Then, $(z_n^1)_{n \in \N}$ converges weakly to a solution to 
		Problem~\ref{prob:opti}.
	\end{prop}
	
	\begin{proof}
		Let $\HH = \H \times \G \times \R^p$, define
		\begin{equation}\label{eq:defoperators}
			\begin{cases}
				A \colon \HH\to 2^{\HH} \colon (x,u,v) \mapsto 
				\partial f (x) \times 
				\partial g^*(u) \times N_{\RP^p}(v),\\
				B_1 \colon \HH\to \HH \colon (x,u,v) \mapsto (\nabla 
				h (x) , 0 , 0),\\
				B_2 \colon \HH\to \HH \colon (x,u,v) \mapsto 
				(M^*u,-Mx,0),\\
			\end{cases}
		\end{equation}
		and consider the operator $B_3$ defined in 
		\eqref{eq:defB3prop}.
		Note that $A$ is maximally monotone \cite[Proposition 20.23  
		\& 
		Proposition 20.25]{bauschkebook2017}, $B_1$ is 
		$\beta$-cocoercive 
		\cite[Corollary 18.17]{bauschkebook2017}, $B_2$ is 
		$\|M\|$-Lipschitzian 
		\cite[Proposition~2.7(ii)]{briceno2011SIAM} \& 
		\cite[Fact~2.20]{bauschkebook2017}, and the operator  $B_3$ 
		is 
		maximally 
		monotone by Proposition~\ref{prop:B3cont}. Furthermore, note 
		that 
		$\dom A= \dom (\partial f)\times \dom(\partial g^*) \times 
		\RP^p$ and 
		$\dom B_3=\left(\cap_{i=1}^n \interior\dom  e_i\right)\times 
		\G 
		\times \RP^p$. Hence, it follows from $\dom (\partial f) 
		\subset  
		\cap_{i=1}^n \interior \dom  e_i $ that $\dom A \cup 
		X \subset \dom B_3$ and $0 \in \interior 
		\left(\dom A-\dom B_3\right)$.  Therefore, $A+B_3$ is 
		maximally monotone 
		\cite[Corollary~25.5(ii)]{bauschkebook2017}. 
		Altogether, the inclusion in \eqref{eq:incluAB} is a 
		particular 
		instance of Problem~\ref{prob:problem1}. Define, for every $n 
		\in 
		\N$, $\boldsymbol{x}_n=(x_n^1,x_n^2,x_n^3)$. Hence,
		\eqref{eq:defx123}, \eqref{eq:defB3prop}, and 
		\eqref{eq:defoperators} 
		yield
		\begin{equation}\label{eq:algorithmnegrita}
			(\forall n\in\N)\quad 
			\begin{array}{l}
				\left\lfloor
				\begin{array}{l}
					\boldsymbol{x}_n= J_{\gamma_n 
						A}(\boldsymbol{z}_n-\gamma_n 
					(B_1+B_2+B_3)\boldsymbol{z}_n)\\
					\boldsymbol{z}_{n+1}= P_X (\boldsymbol{x}_n + 
					\gamma_n 
					(B_2+B_3)\boldsymbol{z}_n-\gamma_n 
					(B_2+B_3)\boldsymbol{x}_n),
				\end{array}
				\right.
			\end{array}
		\end{equation}
		where $\gamma_n$, by \eqref{eq:thetateoop} and 
		\eqref{eq:defoperators}, satisfies
		\begin{equation}\label{eq:thetateodem}
			\gamma \|B_3\boldsymbol{z}_n-B_3 J_{\gamma A} 
			(\boldsymbol{z}_n 
			- \gamma (B_1+B_2+B_3)\boldsymbol{z}_n)\| \leq 
			\theta\|\boldsymbol{z}_n-J_{\gamma A} (\boldsymbol{z}_n - 
			\gamma 
			(B_1+B_2+B_3)\boldsymbol{z}_n)\|.
		\end{equation}
		Note that, if we assume \ref{teo:ophypb}, by 
		Proposition~\ref{prop:B3cont}, $B_3$ is uniformly continuous 
		in 
		every weak compact subset of $\overline{\dom}\partial f 
		\times \overline{\dom} \partial g^* \times \RP^p = 
		\overline{\conv}(\dom A ) = 
		\overline{\conv}(\dom A\cup X)$ \cite[Corollary~21.14 \& 
		Exercise~3.2]{bauschkebook2017}. In view of $0 \in \sri (\dom 
		g - M(\dom f))$,  
		\eqref{e:slater}, and \cite[Proposition 
		27.21]{bauschkebook2017}, there exists a solution
		$\boldsymbol{\hat{z}}=(\hat{x},\hat{u},\hat{v}) 
		\in \H\times \G \times \R^p$ to 
		\eqref{eq:incluAB} such that $\hat{x}$ is a solution to 
		Problem~\ref{prob:opti}. Altogether, since $X 
		\subset \dom A$,
		Theorem~\ref{teo:convergencia} implies that
		$\boldsymbol{z}_n 
		\weak \boldsymbol{\hat{z}}$ and the result follows.\qed
	\end{proof}
	\section{Numerical Experiments}\label{sec:5}
	In this section we consider the following optimization problem
	
	\begin{equation}\label{eq:numericproblem}
		\min_{\substack{ {y}^0 \leq {x}\leq {y}^1\\
				x_i(\ln (x_i/a_i) -1 ) -r_i \le 0, \\
				i \in \{1, \ldots, n\}}}  \alpha \|M x\|_1 + 
		\frac{1}{2}\|Ax-z\|^2,
	\end{equation}
	where $M \in \R	^{{r}\times {n}}$, $A\in \R^{m\times n}$, 
	$z \in \R^m$, ${y}^0=(\eta_i^0)_{1\le i\le n} \in \R^n$, 
	${y}^1=(\eta_i^1)_{1\le i\le N} \in \R^N$, and, for every $i \in 
	\{1, 
	\ldots, p\}$, $r_i \in 
	\left]-a_i 
	, 
	0\right[$ and $a_i \in \RPP$. Set
	\begin{equation}
		\begin{cases}
			C=\times_{i=1}^n[\eta_i^0,\eta_i^1],\\
			f=\iota_{C},\\
			g=\alpha \|\cdot\|_1,\\
			h=\|A\cdot -z\|^2/2,\\
			e=(e_i(\cdot))_{i=1}^n,
		\end{cases}
	\end{equation}
	where
	\begin{equation}
		(\forall i=1,\ldots,n) \quad e_i \colon \R^n \to 
		\RX \colon 
		x \mapsto \begin{cases}
			x_i(\ln (x_i/a_i)-1)-r_i,& \text{ if } x_i > 0;\\
			-r_i,  &\text{ if } x_i = 0 ;\\
			+\infty, &\text{ otherwise}. 
		\end{cases}
	\end{equation}
	Then, we have $f \in \Gamma_0(\R^n)$, $g \in 
	\Gamma_0(\R^r)$, and $\nabla h$ is $\|A\|^2$-Lipschitizian. 
	Additionally, for every $i \in \{1,\ldots,n\}$, $e_i$ is 
	G\^ateaux 
	differentiable in $\RPP^n$,
	\begin{equation}\label{eq:nablaei}
		(\nabla e_i (x))_k=	\begin{cases}
			\ln x_k ,& \text{ if } k =i; \\
			0, & \text{ otherwise},
		\end{cases}
	\end{equation}
	$\dom e_i$ is closed, 
	$\cap_{i=1}^n \dom e_i = 
	\RP^n$, \eqref{e:slater} holds, and 
	\begin{equation*}
		0 \in 
		\interior\left(\dom(\partial f)-\cap_{i=1}^n \dom 
		e_i\right)=\times_{i=1}^N\left]-\infty,\eta_i^1\right[.
	\end{equation*} Hence, 
	the 
	optimization 
	problem in 
	\eqref{eq:numericproblem} is a particular instance of 
	Problem~\ref{prob:opti}. In this setting, since 
	$g^*=\iota_{[-\alpha,\alpha]^r}$ 
	\cite[Example~13.32(v) \& Proposition 
	13.23(i)]{bauschkebook2017}, we 
	consider $X_1\times 
	X_2 \times X_3 = C\times [-\alpha,\alpha]^r \times \RP^p$ in 
	order to 
	write 
	the recurrence in \eqref{eq:algorithmop} as 
	Algorithm~\ref{alg:opti} 
	below.

	\begin{algorithm}
		\caption{} 
		\label{alg:opti}
		\begin{algorithmic}[1]
			\STATE{Fix $\boldsymbol{z}_0=(z_0^1, z_0^2,z_0^3) \in 
				\R^n 
				\times \R^r \times \R^n $. Let $\sigma \in 
				\left]0,1\right[$, let 
				$\varepsilon
				=\|A\|^4 
				\frac{\sqrt{1+16\|M\|^2/\|A\|^4}-1}{8\|M\|^2}$, let 
				$\theta = 
				2\varepsilon\|M\|(1-\s)/\|A\|^2$, and let $\epsilon 
				>0$.}
			\WHILE{ $r_n > \epsilon $}
			\STATE{$\gamma=2\varepsilon \|M\|$}
			\STATE{{\tt V} $=0$} 
			\WHILE{ {\tt V} $=0$}
			\STATE{$\gamma \to \gamma \cdot \sigma$} 
			\STATE{$\Phi^1(\gamma) = 
				P_{[\eta_1,\eta_2]^n}\Big(z^1_n-\gamma 
				\big( A^* (A z^1_n- z) + M^*z^2_n + \sum_{i=1}^p 
				z_{n,i}^3 
				\ln(z_{n,i}^1)\big)\Big)$}
			\STATE{$\Phi^2(\gamma)=\gamma (\id-\prox_{ \alpha \|\cdot 
					\|_1/\gamma}) (z_n^2 / \gamma + M z^1_n)$}
			\STATE{$\Phi^3(\gamma)=P_{\RP^n}\big(z^3_n+\gamma e 
				(z^1_n)\big)$}
			\STATE{$\boldsymbol{\Phi}(\gamma)=(\Phi^1(\gamma),\Phi^2(\gamma),\Phi^3(\gamma))$}
			
			\IF{$ \sum_{i=1}^p  |z_{n,i}^3\ln  (z_{n,i}^1)  - 
				\Phi_{i}^3(\gamma) 
				\ln 
				(\Phi_{i}^1(\gamma))|^2+ \left\| e(z_{n}^1) - 
				e(\Phi^1(\gamma)) 
				\right\|^2\leq \frac{\theta^2}{\gamma^2}  
				\pnorm{\boldsymbol{z}_n - 
					\boldsymbol{\Phi}(\gamma)}^2$ }
			\STATE{{\tt V} $=1$}
			\ENDIF
			\ENDWHILE
			\STATE{$\gamma_n=\gamma$} 
			\STATE{$(x_n^1,x_n^2,x_n^3)=(\Phi^1(\gamma_n),\Phi^2(\gamma_n),\Phi^3(\gamma_n))$}
			\STATE{	$z_{n+1}^1= P_{[\eta_1,\eta_2]^N} \big(x_n^1 + 
				\gamma_n (M^* z_n^2 + \sum_{i=1}^p z_{n,i}^3 
				\ln(z_{n,i}^1) 
				)-\gamma_n (M^* x_n^2 + \sum_{i=1}^p  x_{n,i}^3 
				\ln(x_{n,i}^1) 
				)\big)$ }
			\STATE{	$z_{n+1}^2= P_{[-\alpha,\alpha]^r}(x_n^2 - 
				\gamma_n M 
				z_n^1 + \gamma_n M x_n^1)$ }
			\STATE{	$z_{n+1}^3=P_{\RP^p} (x_n^3 - \gamma_n e(z_n^1) + 
				\gamma_n e  (x_n^1))$ }
			\STATE{	
				$\boldsymbol{z}_{n+1}=(z_{n+1}^1,z_{n+1}^2,z_{n+1}^3)$
			}
			\STATE{ 
				$r_n=\mathcal{R}\big(\boldsymbol{z}_{n+1},
				\boldsymbol{z}_{n}\big)$ }
			\STATE{$n \to n+1$}
			\ENDWHILE 
			\RETURN{$\boldsymbol{z}_{n+1}$}
		\end{algorithmic}
	\end{algorithm}
	We compare Algorithm~\ref{alg:opti} with the algorithm propose 
	in 
	\cite{BricenoDavis2018} called FBHF and with the MATLAB's 
	\texttt{fmincon} (interior point). 
	
	To solve problem in \eqref{eq:probopti} with FBHF algorithm, we 
	consider $X =X_1\times X_2 \times X_3$ and the 
	operators (see \eqref{eq:incluAB} and \cite[Theorem 
	2.3]{BricenoDavis2018})
	
	\begin{equation}\label{eq:opBricenoDavis2018}
		A=\begin{pmatrix}
			\partial f (\hat{x})\\
			\partial g^*(\hat{u})\\
			N_{\RP^p} (\hat{v} )
		\end{pmatrix}, \quad B_1 = \begin{pmatrix}
			\nabla h (\hat{x})\\
			0\\
			0
		\end{pmatrix}, \quad B_2+B_3 = 
		\begin{pmatrix}
			M^* \hat{u} +\sum_{i=1}^p\hat{v}_i \nabla e_i (\hat{x})  
			\\
			-M \hat{x}\\
			-e(\hat{x})
		\end{pmatrix}.
	\end{equation}
	In our numerical experiments, we generate $20$ random 
	realizations of 
	$A$, $M$, $z$, and $r_1,\ldots, r_n$ for dimensions  $n=m \in 
	\{600,900,1200\}$ and $r \in \{n/3, n/2,2n/3\}$. In each 
	realization we define $a_i=9$ for $i\in \{ 1,\ldots,n\}$, $\alpha 
	= 
	0.05$, and
	$y^0 = \hat{y}_0$ and $y^1=\hat{y}_1 + \text{rand}(n)$, where 
	$\hat{y}_1$ and $\hat{y}_2$ satisfies 
	$e(\hat{y}_0)=e(\hat{y}_1)=0$. 
	For Algorithm~\ref{alg:opti} we consider $\sigma = 0.99$. For 
	FBHF we 
	consider $\varepsilon = 0.8$, $\theta= \sqrt{1-\varepsilon}/2$, 
	$\sigma= 0.99$, the maximally monotone operator $A$, the 
	cocoercive 
	operator $B_1$, and the continuous operator $B_1+B_2$ on 
	\eqref{eq:opBricenoDavis2018} (see \cite[Theorem 
	2.3]{BricenoDavis2018}).

	In Table~\ref{t:1} we provide the 
	average time and iterations to achieve a tolerance 
	$\epsilon=10^{-6}$ for the instances mentioned above. We can 
	observe 
	that, for each instance, Algorithm~\ref{alg:opti} is more 
	efficient 
	than the method FBHF and \texttt{fmincon}. 
	Algorithm~\ref{alg:opti} 
	and FBHF are similar in number of iterations, but each iteration 
	of 
	FBHF is more expensive in time than Algorithm~\ref{alg:opti}. 
	This is 
	because FBHF needs, additionally, to activate the operators 
	$M^*$ 
	and 
	$M$ in each line search. This difference is larger as the 
	dimension 
	of the problem increases. Although \texttt{fmincon} needs less 
	iterations than Algorithm~\ref{alg:opti} and FBHF to reach the 
	stop 
	criterion, each iteration is very expensive in CPU time. Indeed, 
	Algorithm~\ref{alg:opti} reaches the stop criterion in 20\% of 
	the 
	time that \texttt{fmincon} takes.
	
	\begin{table}
		{ \caption{ Average time and average number of iterations of 
				20 
				random realizations of problem in 
				\eqref{eq:numericproblem} for 
				Algorithm~\ref{alg:opti}, FBHF, and \texttt{fmincon}. 
				\label{t:1}}
			\begin{center}
				\begin{tabular}{|c|c|c|c|c|c|c|c|}\cline{2-8}
					\multicolumn{1}{c}{ }	& 
					\multicolumn{1}{|c}{$\epsilon=10^{-6}$} & 
					\multicolumn{2}{|c|}{$ 
						N_2=N_1/3$}&\multicolumn{2}{c|}{$N_2=N_1/2$} 
					&\multicolumn{2}{c|}{ $N_2=2N_1/3$}\\  \hline
					$N_1$ & Algorithm &  Av. Time (s)& Av. Iter & Av. 
					Time 
					(s)& Av. 
					Iter & 
					A. Time (s)& A. Iter  \\ \hline  \hline
					\multirow{3}{*}{$600$}& Alg.~\ref{alg:opti} &  
					6.82 & 
					7845  & 9.61
					& 10384 & 15.79 & 16136  \\ \cline{2-8} 
					&FBHF  &  10.47 & 8280  & 14.22 & 10885 & 23.28 & 
					16772 
					\\ 
					\cline{2-8} 
					&\texttt{fmincon}  &  52.52 & 238  & 66.25 & 276 
					& 69.78 
					& 
					251	 \\ \hline  
					\hline
					\multirow{3}{*}{$900$} & Alg.~\ref{alg:opti}
					& 19.26 & 8185 & 28.89  & 11932 & 52.69 & 20653 
					\\ 
					\cline{2-8}
					& FBHF  & 31.28 & 8568 &46.06& 12375 & 84.17 & 
					21757\\ 
					\cline{2-8}
					& \texttt{fmincon} &  256.71 & 350  & 309.33  & 
					408 & 
					292.21 
					& 368	 \\ \hline	
					\hline
					\multirow{3}{*}{$1200$}& Alg.~\ref{alg:opti}  & 
					36.01 & 
					8809 &  
					62.82 & 14490& 110.76 & 24778   \\ 
					\cline{2-8}
					&FBHF  &  59.41 & 9231 &98.60 & 14783 & 174.08 & 
					25633\\ 
					\cline{2-8}
					&\texttt{fmincon}  & 694.86 & 457& 839.06 & 528  
					& 790.66 
					& 
					462	 \\ \hline 
				\end{tabular}
		\end{center}}
	\end{table}

	\section{Conclusions}
	We provide an algorithm for finding a zero of the sum of a 
	maximally 
	monotone, a cocoercive, a monotone-Lipschitzian, and a 
	monotone-continuous operators in a Hilbertian setting which 
	splits 
	their influence in its implementation. The proposed method 
	exploits the intrinsic properties of each operator activating 
	implicitly the maximally monotone operator via its resolvent, 
	explicitly the cocoercive and the monotone-Lipschitzian operator, 
	and for the continuous operator activation a line search 
	procedure is 
	needed. This method generalizes previous results in 
	\cite{BricenoDavis2018,lionsandmercier1979,Tseng2000SIAM} and it 
	is more efficient than other algorithms in the literature when 
	applied to 
	non-linearly constrained convex optimization problems.
	
	\section*{acknowledgements}
		The first author thanks the support of  Centro de 
		Modelamiento 
		Matemático (CMM), ACE210010 and FB210005, BASAL funds for 
		centers of excellence, grant Redes 180032, and grant 
		FONDECYT 1190871 
		from ANID-Chile. The second author thanks the support 
		of ANID-Subdirección de Capital Humano/Doctorado 
		Nacional/2018-21181024 and by the Direcci\'on de Postgrado y 
		Programas from 
		UTFSM through Programa de Incentivos a la Iniciaci\'on 
		Cient\'ifica 
		(PIIC).

\end{document}